\documentclass[a4paper,12pt]{amsart}
\calclayout
\usepackage{amssymb}
\usepackage{amsmath}
\usepackage{amsfonts}
\usepackage[all]{xy}
\usepackage{mathabx}
\usepackage{hyperref}

\newtheorem{thm}{Theorem}[section]
\newtheorem{cor}[thm]{Corollary}
\newtheorem{lem}[thm]{Lemma}

\theoremstyle{definition}

\theoremstyle{remark}

\numberwithin{equation}{section}

\begin{document}

\title{Matrix regular operator space and operator system}

\author{Kyung Hoon Han}

\address{Department of Mathematical Sciences, Seoul National University, San 56-1 ShinRimDong, KwanAk-Gu, Seoul
151-747, Korea}

\email{kyunghoon.han@gmail.com}

\thanks{}



\date{}

\dedicatory{}

\commby{}


\begin{abstract}
We establish a relationship between Schreiner's matrix regular
operator space and Werner's (nonunital) operator system. For a
matrix ordered operator space $V$ with complete norm, we show that
$V$ is completely isomorphic and complete order isomorphic to a
matrix regular operator space if and only if both $V$ and its dual
space $V^*$ are (nonunital) operator systems.
\end{abstract}

\maketitle

\section{Introduction}

The theory of operator spaces has been developed as a
noncommutative counterpart of the theory of Banach spaces. The
order structures of classical Banach spaces have been studied
mostly under the Banach lattice framework. However, the most basic
examples of operator algebras such as $M_n (n \ge 2)$ do not
possess these lattice structures (consider $2 \times 2$ matrices
$0_2$ and $E_{1,2}+E_{2,1}$ and the family $\lambda
E_{1,1}+\lambda^{-1} E_{2,2}$ for $\lambda>0$). Hence, it is
natural to consider order structures that will work in a
noncommutative setting. In this paper, we focus on such order
structures of operator spaces.

Two of the several basics of operator space theory are Ruan's
representation theorem and the duality: every operator space can
be embedded into $B(H)$ completely isometrically and there is a
natural operator space structure on the Banach dual of an operator
space. From this standpoint, there are two definitions of the
order structures of operator spaces. Werner's (nonunital) operator
system corresponds to the representation and Schreiner's matrix
regular operator space to the duality. Usually, an operator system
means a unital involutive subspace of $B(H)$ or its abstract
characterization given by Choi and Effros \cite{CE}, but here we
follow Werner's terminology. In this paper, a (nonunital) operator
system means a matrix ordered operator space which is completely
isomorphic and complete order isomorphic to an involutive subspace
of $B(H)$ or its abstract characterization given by Werner
\cite{W}. For a matrix ordered operator space $V$ with complete
norm, $V$ is matrix regular if and only if its dual space $V^*$ is
matrix regular \cite{S2}.

The category of (nonunital) operator systems contains the class of
$C^*$-algebras and Haagerup's noncommutative $L_p$-spaces
\cite{H}. The category of matrix regular operator spaces contains
the class of $C^*$-algebras and their duals, preduals of von
Neumann algebras, and the Schatten class $\mathcal S_p$ \cite{S1}.

Karn proved that every matrix regular operator space is a
(nonunital) operator system \cite{K}. Since the dual space of a
matrix regular operator space is matrix regular, the dual space of
a matrix regular operator space is also a (nonunital) operator
system. Its converse would be reasonable in the completely
isomorphic context because Werner's (nonunital) operator system is
defined not in a completely isometric sense but in a completely
isomorphic sense. The purpose of this paper is to show that a
matrix ordered operator space $V$ with complete norm is completely
isomorphic and complete order isomorphic to a matrix regular
operator space if and only if both $V$ and its dual space $V^*$
are (nonunital) operator systems.

\section{Preliminaries}

Recall that a complex vector space $V$ is matrix ordered if
\begin{enumerate}
\item $V$ is a $*$-vector space (hence so is $M_n(V)$ for all $n
\ge 1$), \item each $M_n(V), n \ge 1$, is partially ordered by a
(not necessarily proper) cone $M_n(V)^+ \subset  M_n(V)_{sa}$ ,
and \item if $\alpha \in M_{m,n}$, then $\alpha^* M_m(V)^+ \alpha
\subset M_n(V)^+$.
\end{enumerate}
Here the positive cone need not be proper, in other words, it may
be the case that $M_n(V)^+ \cap - M_n(V)^+ \ne \{0\}$. The reason
to exclude the proper condition is due to the fact that the dual
cone of a proper cone need not be proper.

An operator space $V$ is called a matrix ordered operator space
iff $V$ is a matrix ordered vector space and for every $n \in
\mathbb N$,
\begin{enumerate}
\item the $*$-operation is an isometry on $M_n(V)$, and \item the
cones $M_n(V)^+$ are closed.
\end{enumerate}

For a matrix ordered operator space $V$ and its dual space $V^*$,
the positive cone on $M_n(V^*)$ for each $n \in \mathbb N$ is
defined by
$$M_n(V^*)^+=CB(V, M_n) \cap CP(V, M_n).$$ Then the operator space dual
$V^*$ with this positive cone is a matrix ordered operator space
\cite[Corollary 3.2]{S2}.

For a matrix ordered operator space $V$ with complete norm, we say
that $V$ is a matrix regular operator space if for each $n \in
\mathbb N$ and for all $v \in M_n(V)_{sa}$
\begin{enumerate}
\item $u \in M_n(V)^+$ and $-u \le v \le u$ imply that $\|v\|_n
\le \|u\|_n$, and \item $\|v\|_n<1$ implies that there exists $u
\in M_n(V)^+$ such that $\|u\|_n<1$ and $-u \le v \le u$.
\end{enumerate}
Due to condition (1), it is easily seen that the positive cone of
a matrix regular operator space is always proper. A matrix regular
operator space can be described in another way. A matrix ordered
operator space $V$ with complete norm is matrix regular if and
only if the following condition holds: for all $x \in M_n(V),
\|x\|_n<1$ if and only if there exist $a, d \in M_n(V)^+,
\|a\|_n<1$ and $\|d\|_n<1$, such that $\begin{pmatrix} a & x
\\ x^* & d \end{pmatrix} \in M_{2n}(V)^+$ \cite[Theorem 3.4]{S2}.

The class of matrix regular operator spaces has a nice duality
property. Let $V$ be a matrix ordered operator space with complete
norm. Schreiner showed that $V$ is matrix regular if and only if
its dual space $V^*$ is matrix regular \cite[Corollary 4.7,
Theorem 4.10]{S2}.

Let $X$ be a matrix ordered operator space with a proper positive
cone. For $x \in M_n(X)$, the modified numerical radius is defined
by $$\nu_X(x) = \sup \{|\varphi(\begin{pmatrix} 0 & x \\
x^* & 0 \end{pmatrix})| : \varphi \in M_{2n}(X)^*_{1,+} \}.$$ We
call a matrix ordered operator space with a proper positive cone
an operator system iff there is a $k>0$ such that for all $n \in
\mathbb N$ and $x \in M_n(X)$,
$$\|x\|_n \le k \nu_X(x).$$ Since we always have the inequality
$\nu_X(x) \le \|x\|_n$, we can say that an operator system is a
matrix ordered operator space such that the operator space norm
and the modified numerical radius are equivalent uniformly for all
$n \in \mathbb N$.

Werner showed that $X$ is an operator system if and only if there
is a complete order isomorphism $\Phi$ from $X$ onto an involutive
subspace of $B(H)$, which is a complete topological
onto-isomorphism \cite[Theorem 4.15]{W}. Hence, the operator
system is an abstract characterization of the involutive subspace
of $B(H)$ in a completely isomorphic and complete order isomorphic
sense.

\section{Matrix regular operator space and operator system}

Karn showed that every matrix regular operator space can be
embedded into $B(H)$ 2-completely isomorphically and complete
order isomorphically \cite{K}. Here we give another proof of the
result that is more fitting to Werner's axiomatic framework. This
idea also appears in the recent preprint \cite[Theorem 2.4]{KPTT}.

\begin{thm} \label{main}
Every matrix regular operator space is an operator system with a
dominating constant $2$.
\end{thm}

\begin{proof}
Since the dual space of a matrix regular operator space is matrix
regular and the canonical inclusion map from a matrix ordered
operator space into its bidual is a completely isometric complete
order isomorphism \cite[Theorem 4.9]{S2}, it is sufficient to show
that the dual space of a matrix regular operator space is an
operator system.

Suppose that $V$ is a matrix regular operator space. We choose an
element $F = [f_{ij}]$ in $M_n(V^*)$. Its norm can be written as
$$\begin{aligned} \|F\|_{M_n(V^*)} & = \|F\|_{CB(V,M_n)} \\ & = \sup
\{ \|F_n(x)\|_{M_{n^2}} : x \in M_n(V)_{\|\cdot\|<1} \} \\ & =
\sup \{ \| \begin{pmatrix} 0 & F_n(x) \\ F_n(x)^* & 0
\end{pmatrix} \|_{M_{2n^2}} : x \in M_n (V)_{\|\cdot\|<1} \} \\ &
= \sup \{ | \langle \begin{pmatrix} 0 & F_n(x) \\ F_n(x)^* & 0
\end{pmatrix} \xi | \xi \rangle| : x \in M_n(V)_{\|\cdot\|<1}, \xi
\in (\ell^2_{2n^2})_1 \}. \end{aligned}$$ We choose elements $x$
in $M_n(V)_{\|\cdot \|<1}$ and $\xi$ in $(\ell^2_{2n^2})_1$. By
\cite[Theorem 3.4]{S2}, there exist $a,d$ in $M_n(V)$ such that
$\|a\|_n<1, \|d\|_n<1$ and $\begin{pmatrix} a & x \\ x^* & d
\end{pmatrix} \in M_{2n}(V)^+$. Then we have $\|\begin{pmatrix} a & x \\ x^* & d
\end{pmatrix}\|_{2n} < 2$. We define a linear functional
$\varphi_{x,\xi} : M_{2n}(V^*) \to \mathbb C$ by $$\varphi_{x,\xi}
(G) = {1 \over 2} \langle
\begin{pmatrix} I_{n^2}&0&0&0 \\ 0&0&0&I_{n^2} \end{pmatrix} G_{2n}
(\begin{pmatrix} a&x \\ x^*&d \end{pmatrix})
\begin{pmatrix} I_{n^2}&0\\0&0\\0&0\\0&I_{n^2}
\end{pmatrix} \xi | \xi \rangle,\quad G \in M_{2n}(V^*)=CB(V, M_{2n})$$ Then $\varphi_{x,\xi}$
is a positive contractive functional. Putting $G=\begin{pmatrix}
0&F \\ F^*&0 \end{pmatrix}$, we get
$$\begin{aligned} \varphi_{x,\xi}(\begin{pmatrix} 0&F \\ F^*&0 \end{pmatrix}) & = {1
\over 2} \langle \begin{pmatrix} I_{n^2}&0&0&0 \\ 0&0&0&I_{n^2}
\end{pmatrix} \begin{pmatrix} 0&0&F_n(a)&F_n(x) \\
0&0&F_n(x^*)&F_n(d) \\ F_n^*(a)&F_n^*(x)&0&0 \\
F_n(x)^*&F_n^*(z)&0&0 \end{pmatrix}
\begin{pmatrix}I_{n^2}&0\\0&0\\0&0\\0&I_{n^2} \end{pmatrix} \xi | \xi \rangle \\ & = {1 \over 2} \langle \begin{pmatrix} 0 &
F_n(x) \\ F_n(x)^* & 0 \end{pmatrix} \xi | \xi \rangle.
\end{aligned}$$
It follows that $$\begin{aligned} \|F\|_{M_n(V^*)} & = 2 \sup \{
|\varphi_{x,\xi}(\begin{pmatrix} 0 & F \\ F^* & 0 \end{pmatrix})|
: x \in M_n(V)_{\|\cdot\|<1}, \xi \in (\ell^2_{2n})_1 \} \\ & \le
2 \nu_{V^*}(F),\end{aligned},$$ where $\nu_{V^*}(F)$ denotes the
modified numerical radius of $F$.
\end{proof}

In general, the embedding cannot be chosen completely
isometrically and  complete order isomorphically as can be seen
from the two-dimensional $L^1$-space $\ell^1_2$ \cite[Proposition
1.1]{BKNW}. In the case of the Schatten class $\mathcal S_p$, the
constant can be chosen to be $2^{1 \over p}$ \cite{H}.

The direct converse of Theorem \ref{main} is false. If we consider
the operator system $$\{ \begin{pmatrix} 0 & \alpha
\\ \beta & 0 \end{pmatrix} \in M_{2n} : \alpha, \beta \in M_{n}
\},$$ then its positive cone is trivial, thus the second condition
on matrix regularity cannot be satisfied for any operator system
complete order isomorphic to the above one. Because the dual space
of a matrix regular operator space is also matrix regular, the
matrix regularity of $V$ implies that both $V$ and its dual space
$V^*$ are operator systems. Our next goal is to prove the converse
in a completely isomorphic and complete order isomorphic sense.
Informally, the first and the second conditions on matrix
regularity imply that the positive cone is small and large,
respectively, in some sense. In other words, we can say that the
positive cone of a matrix regular operator space is just the right
size. The definition of an operator system means the dual cone is
large enough, or equivalently, that the positive cone of an
operator system is small enough. As we have just seen, the
positive cone of an operator system may be trivial. We see from
these informal observations that it is natural to consider the
problem set forth as the goal of the present paper.

\begin{lem}\label{regular}
Suppose that $V$ is a matrix ordered operator space with complete
norm satisfying the following two conditions:
\begin{enumerate} \item for all $x, y \in M_n(V)_{sa}, -y \le x \le y$ implies
$\|x\|_n \le \|y\|_n$. \item for all $x \in M_n(V)_{sa}$ with
$\|x\|<1$, there exist $a, d \in M_n(V)^+$ such that $\|a\|_n,
\|d\|_n<K$ and $\begin{pmatrix} a&x \\ x^* &d \end{pmatrix} \in
M_{2n}(V)^+$.
\end{enumerate}
Then $V$ is $K$-completely isomorphic and complete order
isomorphic to a matrix regular operator space.
\end{lem}

\begin{proof}
We first define
$$\| x \|_{reg} = \inf \{ \max \{ \|a\|_n, \|d\|_n \} : \begin{pmatrix} a & x \\
x^* & d \end{pmatrix} \in M_{2n}(V)^+ \}, \qquad x \in M_n(V).$$
Note that this definition is similar to the norm $\|\cdot\|_{dec}$
of a decomposable map \cite{Ha}. The set which we take an infimum
over is not empty and we have $\|x\|_{reg} \le K \|x\|_n$.
Multiplying both sides by the scalar matrix $\begin{pmatrix} 1&0
\\ 0&-1
\end{pmatrix}$, we see that
$$\begin{pmatrix} a & x \\ x^* & d \end{pmatrix} \in M_{2n}^+ \quad \text{if
and only if} \quad -\begin{pmatrix} a&0 \\ 0&d \end{pmatrix} \le
\begin{pmatrix} 0&x \\ x^* &0 \end{pmatrix} \le \begin{pmatrix} a&0 \\ 0&d \end{pmatrix}.$$
The inequality $\|x\|_n \le \|x\|_{reg}$ follows from condition (1). We choose elements $\begin{pmatrix} a_1 & x \\
x^* & d_1 \end{pmatrix}$ and $\begin{pmatrix} a_2 & y \\ y^* & d_2
\end{pmatrix}$ in $M_{2n}(V)^+$. Since $\begin{pmatrix} a_1 + a_2 & x+y \\
(x+y)^* & d_1+ d_2 \end{pmatrix}$ belongs to $M_{2n}(V)^+$, we
have
$$\begin{aligned} \|x+y \|_{reg} & \le \max \{ \|a_1+a_2\|_n,
\|d_1+d_2\|_n\} \\ & =  \| \begin{pmatrix} a_1 + a_2 & 0 \\ 0 &
d_1+d_2 \end{pmatrix}\|_{2n} \\ & \le \| \begin{pmatrix} a_1 & 0 \\
0 & d_1 \end{pmatrix} \|_{2n} + \|\begin{pmatrix} a_2 & 0 \\ 0 &
d_2 \end{pmatrix}\|_{2n} \\ & = \max \{ \|a_1\|_n, \|d_1\|_n \} +
\max \{ \| a_2 \|_n, \|d_2\|_n \}.
\end{aligned}$$
It follows that $\|x+y\|_{reg} \le \|x\|_{reg} + \|y\|_{reg}$. For
$\lambda = e^{it}|\lambda| \in \mathbb C$, we have
$$\begin{pmatrix} |\lambda|a & \lambda x \\ (\lambda x)^* &
|\lambda| b \end{pmatrix} = \begin{pmatrix} e^{it} |\lambda|^{1
\over 2} & 0 \\ 0 & |\lambda|^{1 \over 2} \end{pmatrix}
\begin{pmatrix} a & x \\ x^* & d \end{pmatrix} \begin{pmatrix} e^{-it} |\lambda|^{1
\over 2} & 0 \\ 0 & |\lambda|^{1 \over 2} \end{pmatrix} \in
M_{2n}(V)^+.$$ It follows that $\|\lambda x\|_{reg} = |\lambda|
\|x\|_{reg}$. Because we have $\|x\|_n \le \|x\|_{reg}$ for all $x
\in M_n(V)$, $\|x\|_{reg}=0$ implies $x=0$. Hence, $\|\cdot
\|_{reg}$ is a norm on $M_n(V)$ for each $n \in \mathbb N$.

Next let us show that $\|\alpha x \beta \|_{reg} \le \|\alpha \|
\|x\|_{reg} \|\beta\|$ for $x \in M_m(V), \alpha \in M_{n,m},
\beta \in M_{m,n}$. To this end, we may assume that $\|\alpha\| =
\|\beta\|$. For $\begin{pmatrix} a & x \\ x^* & d
\end{pmatrix} \in M_{2m}(V)^+$, we have $$\begin{pmatrix} \alpha a
\alpha^* & \alpha x \beta \\ (\alpha x \beta)^* & \beta^* d \beta
\end{pmatrix} = \begin{pmatrix} \alpha & 0  \\ 0 & \beta^* \end{pmatrix} \begin{pmatrix} a & x
\\ x^* & d \end{pmatrix} \begin{pmatrix} \alpha^* & 0 \\ 0 & \beta \end{pmatrix} \in M_{2n}(V)^+.$$
It follows that $$\begin{aligned} \|\alpha x \beta \|_{reg} & \le
\max \{ \|\alpha a \alpha^* \|_n, \|\beta^* d\beta \|_n \} \\ &
\le \|\alpha\| \|\beta\| \max \{ \|a\|_m, \|d\|_m
\},\end{aligned}$$ thus $\|\alpha x \beta \|_{reg} \le \|\alpha\|
\|\beta\| \|x\|_{reg}$. Suppose that $$\begin{pmatrix} a_1 & x \\
x^* & d_1
\end{pmatrix} \in M_{2m}(V)^+ \quad \text{and} \quad \begin{pmatrix} a_2 & y \\ y^* & d_2
\end{pmatrix} \in M_{2n}(V)^+$$ with
$\|a_1\|_m,\|d_1\|_m<\|x\|_{reg}+\varepsilon$ and
$\|a_2\|_n,\|d_2\|_n<\|y\|_{reg}+\varepsilon.$ Then we have
$$\begin{pmatrix} a_1 & 0 & x & 0 \\ 0 & a_2 & 0 & y \\ x^* & 0 &
d_1 & 0\\ 0 & y^* & 0 & d_2 \end{pmatrix} = \begin{pmatrix}
1&0&0&0 \\ 0&0&1&0 \\ 0&1&0&0 \\ 0&0&0&1 \end{pmatrix}
\begin{pmatrix} a_1 & x & 0 & 0 \\ x^* & d_1 & 0 & 0 \\ 0 & 0& a_2 & y \\ 0 & 0 & y^* & d_2
\end{pmatrix} \begin{pmatrix} 1&0&0&0 \\ 0&0&1&0 \\ 0&1&0&0 \\ 0&0&0&1 \end{pmatrix} \in
M_{2(m+n)}(V)^+.$$ It follows that $$\| x \oplus y \|_{reg} \le
\max \{ \|a_1\|_m, \|a_2\|_n, \|d_1\|_m, \|d_2\|_n \} < \max \{
\|x\|_{reg}, \|y\|_{reg} \} + \varepsilon.$$ Hence, $(V,
\|\cdot\|_{reg})$ is an operator space.

For $\begin{pmatrix} a & x \\ x^* & d \end{pmatrix} \in
M_{2n}(V)^+$, we have $$\begin{pmatrix} d & x^* \\ x & a
\end{pmatrix} = \begin{pmatrix} 0 & 1 \\ 1 & 0 \end{pmatrix}
\begin{pmatrix} a & x \\ x^* & d \end{pmatrix} \begin{pmatrix} 0 & 1 \\ 1 & 0
\end{pmatrix} \in M_{2n}(V)^+,$$ thus $\|x\|_{reg}=\|x^*\|_{reg}$. Since the identity map
$id : (V, \|\cdot\|) \to (V, \|\cdot\|_{reg})$ is a $K$-complete
isomorphism, the operator space $(V, \|\cdot\|_{reg})$ is complete
and the positive cone $M_n(V)^+$ is closed with respect to the
norm $\|\cdot\|_{reg}$. For $a \in M_n(V)^+$, we have
$$\begin{pmatrix} a & a \\ a & a
\end{pmatrix} = \begin{pmatrix} 1 \\ 1 \end{pmatrix} a
\begin{pmatrix} 1 & 1 \end{pmatrix} \in M_{2n}(V)^+,$$ thus
$\|a\|_{reg}=\|a\|_n$. If an element $x$ belongs to $M_n(V)$ with
$\|x\|_{reg}<1$, then there exist elements $a$ and $d$ in
$M_n(V)^+$ such that $\|a\|_n, \|d\|_n<1$ and $\begin{pmatrix} a &
x \\ x^* & d \end{pmatrix} \in M_{2n}(V)^+$. Since $a$ and $d$ are
positive, we have $\|a\|_{reg}, \|d\|_{reg} <1$. The converse is
obvious. By \cite[Theorem 3.4]{S2}, $(V, \|\cdot\|_{reg},
\{M_n(V)^+\}_{n \in \mathbb N}$) is a matrix regular operator
space. The identity map $id : (V, \|\cdot\|,\{M_n(V)^+\}_{n \in
\mathbb N}) \to (V, \|\cdot\|_{reg}, \{M_n(V)^+\}_{n \in \mathbb
N})$ is a $K$-completely isomorphic complete order isomorphism.
\end{proof}

\begin{lem}\label{dual}
Suppose that a matrix ordered operator space $V$ and its dual
space $V^*$ are operator systems. Then the dual space $V^*$ is
completely isomorphic and complete order isomorphic to a matrix
regular operator space.
\end{lem}

\begin{proof}
By looking at the dual space $V^*$ as a $*$-subspace of $B(K)$, we
can say that condition (1) of Lemma \ref{regular} is satisfied.
Suppose that $W$ is a $*$-subspace of $B(H)$ and $\Phi : W \to V$
is a completely isomorphic complete order isomorphism with
$\|\Phi\|_{cb}\le 1$. We put $$X = \{ \begin{pmatrix} \lambda I_H
& x \\ y & \mu I_H
\end{pmatrix} : \lambda, \mu \in \mathbb C, x,y \in W \}.$$
Then $X$ is a unital operator system in $B(H^2)$. We take an
element $F$ in $M_n(V^*)$ with $\|F\|_{M_n(V^*)}<1$. By
\cite[Lemma 8.1]{P}, the linear map $\varphi : X \to M_{2n}$
defined by $$\varphi (\begin{pmatrix} \lambda I_H& x \\ y & \mu
I_H \end{pmatrix}) = \begin{pmatrix} \lambda I_n & F \circ \Phi (x) \\
F^* \circ \Phi(y) & \mu I_n \end{pmatrix}$$ is a unital completely
positive map. By Arveson's extension theorem \cite{A}, $\varphi$
has a unital completely positive extension $\psi : M_2(W)+ \mathbb
C I_H \oplus \mathbb C I_H \to M_{2n}.$ The linear map $$\theta :
x \in W \mapsto \begin{pmatrix} x&x \\ x&x \end{pmatrix} =
\begin{pmatrix} 1\\1 \end{pmatrix} x \begin{pmatrix} 1&1
\end{pmatrix} \in M_2(W)$$ is completely positive. We write
$$\psi \circ \theta = \begin{pmatrix} \varphi_1 & F \circ \Phi \\ F^*
\circ \Phi & \varphi_2 \end{pmatrix} \in M_{2n}(W^*)^+,\qquad
\|\varphi_1\|_{M_n(W^*)} \le 1, \|\varphi_2\|_{M_n(W^*)} \le 1.$$
Then we have $$\begin{pmatrix} \varphi_1 \circ \Phi^{-1} & F \\
F^* & \varphi_2 \circ \Phi^{-1} \end{pmatrix} \in
M_{2n}(V^*)^+\quad \text{and} \quad \|\varphi_1 \circ
\Phi^{-1}\|_{cb}, \|\varphi_2 \circ \Phi^{-1} \|_{cb} \le
\|\Phi^{-1}\|_{cb}.$$ By Lemma \ref{regular}, we conclude that the
dual space $V^*$ is completely isomorphic and complete order
isomorphic to a matrix regular operator space.
\end{proof}

\begin{thm}
Suppose that both $V$ and its dual space $V^*$ are operator
systems with complete norm. Then $V$ is completely isomorphic and
complete order isomorphic to a matrix regular operator space.
\end{thm}

\begin{proof}
Once again we consider the space $V$ as a $*$-subspace of $B(H)$
and see that condition (1) of Lemma \ref{regular} is satisfied. By
Lemma \ref{dual}, there exist a matrix regular operator space $W$
and a completely isomorphic complete order isomorphism $\Phi : V^*
\to W$. For $x \in M_n(V)$, we define
$$\|x\|_{M_n(W_*)} = \sup \{ |\varphi(x)| : \varphi \in M_n(V)^*, \|\Phi(\varphi)\|_{T_n(W)} \le 1 \}.$$
Endowing $V$ with this matrix norm $\|\cdot\|_{W_*}$, we get an
operator space predual of $W$. Since the two norms
$\|\cdot\|_{M_n(W_*)}$ and $\|\cdot\|_{M_n(V)}$ are equivalent,
the positive cone $M_n(V)^+$ is closed with respect to the norm
$\|\cdot\|_{M_n(W_*)}$ and the operator space $(V,
\|\cdot\|_{W_*})$ is complete. Applying the conjugate linear
variation of \cite[Theorem 4.1.8]{ER} to the involution, we get
$$\begin{aligned} \|x^*\|_{M_n(W_*)} & = \sup \{ |\varphi(x^*)| :
\|\Phi(\varphi)\|_{T_n(W)} \le 1 \} \\ & = \sup \{ |\varphi(x)| :
\|\Phi(\varphi)^*\|_{T_n(W)} \le 1 \} \\ & = \|x\|_{M_n(W_*)}
\end{aligned}$$ Hence the space $W_* := (V,\{ \|\cdot\|_{M_n(W_*)} \}_{n \in \mathbb N},\{ M_n(V)^+ \}_{n \in \mathbb N})$
is a matrix ordered operator space with a complete norm, and its
dual space is $W$. Since the predual of a matrix regular operator
space is also matrix regular, $W_*$ is matrix regular. The predual
map $\Phi_* : W_* \to V$ is completely isomorphic and complete
order isomorphic.
\end{proof}

\begin{cor}
 For a matrix ordered operator space $V$ with complete norm,
 $V$ is completely isomorphic and complete order isomorphic
to a matrix regular operator space if and only if both $V$ and its
dual space $V^*$ are operator systems.
\end{cor}



\begin{thebibliography}{AAA}
\bibitem[A]{A} W. Arveson, \textit{Subalgebras of $C^*$-algebras},
Acta Math. \textbf{123} (1969), 141--224.

\bibitem[BKNW]{BKNW} D. P. Blecher, K. Kirkpatrick, M. Neal and W.
Werner, \textit{Ordered involutive operator spaces}, Positivity
\textbf{11} (2007), no. 3, 497--510.

\bibitem[CE]{CE} M.-D. Choi and E. Effros, \textit{Injectivity and
operator spaces}, J. Funct. Anal. \textbf{24} (1977), 156--209.

\bibitem[ER]{ER} E. Effros and Z-J. Ruan, \textit{Operator Spaces}, London Math. Soc.
Monographs, New Series \textbf{23}, Oxford University Press, New
Yorkk, 2000.

\bibitem[Ha]{Ha} U. Haagerup, \textit{Injectivity and decomposition of completely
bounded maps}, Operator Algebras and their Connections with
Topology and Ergodic Theory (Bu\c{s}teni 1983), Lecture Notes in
Math. \textbf{1132}, Springer-Verlag, Berlin 1985, pp 170--222.

\bibitem[H]{H} K. H. Han, \textit{Noncommutative $L_p$-space and
operator system}, Proc. Amer. Math. Soc. \textbf{137} (2009),
4157--4167.

\bibitem[K]{K} A. K. Karn, \textit{Corrigendum to the paper: ``Adjoining an order unit to a matrix ordered space''},
Positivity \textbf{11} (2007), no. 2, 369--374

\bibitem[KPTT]{KPTT} A. Kavruk, V. I. Paulsen, I. G. Todorov and M.
Tomforde, \textit{Tensor products of operator systems}, preprint.

\bibitem[P]{P} V. I. Paulsen, \textit{Completely Bounded Maps and
Operator Algebras}, Cambridge Studies in Advanced Mathematics, 78.
Cambridge University Press, Cambridge, UK, 2002.

\bibitem[S1]{S1} W. J. Schreiner, \textit{Matrix Regular Orders on Operator
Spaces}, Ph.D. thesis, University of Illinois at Urbana-Champaign,
1995.

\bibitem[S2]{S2} W. J. Schreiner, \textit{Matrix regular operator spaces},
J. Funct. Anal. \textbf{152} (1998), 136--175.

\bibitem[W]{W} W. Werner, \textit{Subspaces of $L(H)$ that are
$*$-invariant}, J. Funct. Anal. \textbf{193} (2002), 207--223.
\end{thebibliography}
\end{document}